\documentclass[a4paper]{amsart}

\usepackage{amsmath,amscd,amsthm,amssymb,xypic}
\usepackage{color}
\usepackage{tikz-cd}
\usepackage{tikz}
\usepackage{enumerate}
\usepackage{relsize}
\usepackage{comment}
\usepackage{mathtools}

\newtheorem{theo}{Theorem}[section]
\newtheorem{lemm}[theo]{Lemma}
\newtheorem{prop}[theo]{Proposition}

\theoremstyle{definition}
\newtheorem{defi}[theo]{Definition}
\newtheorem{rema}[theo]{Remark}

\newtheorem{exam}[theo]{Example}

\newcommand{\bma}{\left(\begin{matrix}}
\newcommand{\ema}{\end{matrix}\right)}

 \usepackage[percent]{overpic}

\newcommand{\CC}{\mathbb{C}}

\newcommand{\QQ}{\mathbb{Q}}
\newcommand{\RR}{\mathbb{R}}

\newcommand{\TT}{\mathbb{T}}

\newcommand{\ZZ}{\mathbb{Z}}
\newcommand{\Aa}{\mathcal{A}}

\newcommand{\Ff}{\mathcal{F}}
\newcommand{\Gg}{\mathcal{G}}
\newcommand{\Hh}{\mathcal{H}}

\newcommand{\kk}{\mathbf{k}}

\newcommand{\del}{\partial}
\newcommand{\delb}{{\bar \partial}}

\newcommand{\lra}{\longrightarrow}

\renewcommand{\l}{\ell}

\newcommand{\ov}{\overline}

\newcommand{\Ker}{\mathrm{Ker}}

\newcommand{\Img}{\mathrm{Im}}
\newcommand{\Dec}{\mathrm{Dec}}

\begin{document}

\title{Filtered $A$-infinity structures in complex geometry}

\author{Joana Cirici}
\address{Departament de Matem\`{a}tiques i Inform\`{a}tica, Universitat de Barcelona,
Gran Via 585,
08007 Barcelona, Spain /
Centre de Recerca Matem\`{a}tica, Edifici C, Campus Bellaterra, 08193
Bellaterra, Spain
}

\email{jcirici@ub.edu}

\thanks{J. Cirici acknowledges the Serra H\'{u}nter Program. Her work was also partially funded by the 
Spanish State Research Agency (Mar\'{i}a de Maeztu Program CEX2020-001084-M and I+D+i project PID2020-
117971GB-C22/MCIN/AEI/10.13039/501100011033) as well as by the French National Research Agency
(ANR-20-CE40-0016). 
}

\author{Anna Sopena-Gilboy}
\address{Departament de Matem\`{a}tiques i Inform\`{a}tica,  Universitat de Barcelona\\
Gran Via 585\\
08007 Barcelona, Spain }
\email{asopenagilboy@ub.edu}


\date{}

\dedicatory{}

%   "Communicated by" -- provide editor's name; required.
\commby{}

%    Abstract is required.

\begin{abstract}
We prove a filtered version of the Homotopy Transfer Theorem which gives an $A$-infinity algebra structure on any page of the spectral sequence associated to a filtered dg-algebra. We then develop various applications to the study of the geometry and topology of complex manifolds, using the Hodge filtration, as well as to complex algebraic varieties, using mixed Hodge theory.
\end{abstract}

\maketitle

\section{Introduction}
On a complex manifold, the Hodge filtration gives a spectral sequence connecting Dolbeault cohomology (a geometric invariant) with de Rham cohomology (a topological invariant). In particular, this spectral sequence carries both geometric and topological information. The algebra structure on the de Rham complex induces naturally an algebra structure on any page of the associated spectral sequence. However, this structure is not topological, not even at the $E_\infty$-page. The main motivation of the present work is to understand how the
topological algebraic structure may
be recovered at any page of the associated spectral sequence.

Given a cochain complex defined over a field, there is always a homotopy transfer diagram connecting it with its cohomology. 
By Kadeishvili's theory \cite{Kade}, when the complex is a dg-algebra, this diagram induces a natural $A_\infty$-structure on its cohomology with trivial differential, which is unique up to isomorphism. This $A_\infty$-structure is strongly related to Massey products and serves as a powerful homotopical invariant of the original dg-algebra.
A first main result of this paper is Theorem \ref{Kadeishvili1filtrat}, which gives a homotopy transfer diagram connecting a filtered cochain complex with the first stage of its associated spectral sequence. This induces naturally on $E_1$ an $A_\infty$-structure unique up to isomorphism. Although the differential of this $A_\infty$-structure is not trivial, it satisfies a certain minimality condition in the filtered setting and in Theorem \ref{r0} we show that every filtered dg-algebra has a filtered minimal $A_\infty$-model.
The existence and uniqueness of filtered minimal models is brought to any page of the associated spectral sequence in Theorem \ref{higherr}.
All results are also true in the commutative and Lie settings, giving $C_\infty$ and $L_\infty$ structures on spectral sequences when the initial filtered complex has a compatible structure of a commutative or a Lie dg-algebra.

The interaction between spectral sequences and $A_\infty$-structures was previously studied by Lapin (see \cite{Lapinprimer} and related works) and Herscovich \cite{Herscovich} via formal deformation theory. Here we propose a different strategy, mimicking the approach of 
Markl \cite{Markl2:06:geophys}.
This allows us to control the behavior of higher operations with respect to filtrations as well as to inductively compute filtered $A_\infty$-models on various examples of geometric origin.

We apply the above filtered homotopy theory to the Hodge filtration on the complex de Rham algebra of a complex manifold. We compute new geometric-topological invariants for complex manifolds and study multiplicative structures on the Fr\"{o}licher spectral sequence. We introduce and compare \textit{Dolbeault $A_\infty$-structures} and \textit{Hodge-de Rham $A_\infty$-structures} on various examples.
As is well-known, on compact K\"{a}hler manifolds the Fr\"{o}licher spectral sequence degenerates always at the first stage and the higher multiplicative structures on cohomology are trivial as a consequence of the Formality Theorem of \cite{DGMS}. In a broader study of this phenomenon, we consider multiplicative structures on mixed Hodge complexes, functorially defined for any complex algebraic variety. We prove strictness of higher operations with respect to both the Hodge and weight filtrations, as well as a ``purity implies formality'' statement in the $A_\infty$-context, which applies to complex algebraic varieties whose weight filtration is pure.

\subsection*{Acknowledgments}
We would like to thank Nero Budur for encouraging us to develop Section \ref{SecMHT} and Daniele Angella, Geoffroy Horel, Muriel Livernet and Scott Wilson for useful discussions. Thanks also to the referee for the comments.

\section{Filtered homotopy transfer diagrams}

We will consider filtered complexes $(A,d,F)$ defined over a field $\kk$, with bounded below and exhaustive decreasing filtrations:
\[0\subset\cdots\subseteq F^{p+1}A\subseteq F^pA\subseteq \cdots\subseteq A\quad\text{ with }\quad
d(F^pA^n)\subseteq F^pA^{n+1}.\]
Recall that the differential $d$ of $A$ is said to be \textit{strictly compatible with filtrations}
if \[d(F^pA)=F^pA\cap dA.\]
This happens if and only if its associated spectral sequence degenerates at $E_1$. 

% More generally, 
% the spectral sequence associated to a filtered complex $(A,d,F)$ degenerates at $E_r$, with $r\geq 1$, if
% and only if its differential satisfies 
% \[d(F^pA^n)= F^{p+r-1}A^{n+a}\cap d(A).\]

\begin{defi}
 A morphism of filtered complexes $f:A\to B$ is said to be a \textit{filtered quasi-isomorphism} if the restriction 
 $F^pf:F^pA\to F^pB$
 is a quasi-isomorphism for all $p$, so that $H^*(F^pf)$ is an isomorphism.
\end{defi}

\begin{rema}\label{equivdefs}
The condition that $H^*(F^pf)$ is an isomorphism for all $p$ is equivalent to asking that the \textit{filtered mapping cone} $(F^pC(f),D)$ of $f$, given by
\[F^pC^n(f)=F^pA^{n+1}\oplus F^pB^n\quad\text{with}\quad D(a,b)=(-da,-fa+db)\]
is acyclic: $H^n(F^pC(f))=0$ for all $p$ and $n$.
For complexes with bounded below filtrations, this is equivalent to asking that
$E_1^{p,q}(f)\cong H^{p+q}(Gr^p_Ff)$ is an isomorphism for all $p,q$ or, equivalently, that $Gr^p_FC(f)$ is acyclic for all $p$.
\end{rema}

\begin{defi}
Let $f,g:(A,d,F)\rightarrow (B,d,F)$ be two morphisms of filtered complexes. 
A \textit{filtered homotopy} from $f$ to $g$ is a map $h:A\rightarrow B[-1]$ such that
\[h(F^pA^n)\subseteq F^{p}B^{n-1}\quad\text{for all $p$ and all $n$, and}\quad hd+dh=f-g.\]
\end{defi}

Filtered homotopies behave as expected with respect to spectral sequences:

\begin{lemm}Let $h$ be a filtered homotopy between two maps of filtered complexes $f,g:(A,d,F)\to (B,d,F)$. The induced maps $f_r,g_r:E_r(A)\rightarrow E_r(B)$ satisfy \[f_r=g_r\quad\text{ for }\quad r>0\quad\text{ and }\quad f^*=g^*:H^*(A)\rightarrow H^*(B).\]
\end{lemm}

In particular, every filtered homotopy equivalence is a filtered quasi-isomorphism.
We now have all the ingredients for a homotopy transfer in the filtered context:

\begin{defi}\label{deffilHTD}
A \textit{filtered homotopy transfer diagram} is given by the data
\[
\begin{tikzcd}[ampersand replacement = \&]
(A,d,F) \arrow[r, rightarrow, shift left, "f"] \arrow[r, shift right, leftarrow, "g" swap] \arrow[loop left, "h"] \& (B,d,F)
\end{tikzcd}\quad\quad\text{ where:}
\]
\begin{enumerate}[(1)]
    \item $(A,d,F)$ and $(B,d,F)$ are filtered complexes,
    \item $f$ and $g$ are morphisms of filtered complexes, and
    \item $h$ is a filtered homotopy from  the identity to $gf$.
\end{enumerate}
\end{defi}

We next recall the classical result on the existence of homotopy transfer diagrams for complexes of vector spaces
(see for instance Lemma 9.4.7 in \cite{LV}). We include the proof for completeness and as a comparison with the filtered version.
 
 \begin{theo}\label{Kadeclassic}
 Let $(A,d)$ be a complex of vector spaces. There is a homotopy transfer diagram  
\[
\begin{tikzcd}[ampersand replacement = \&]
(A,d) \arrow[r, rightarrow, shift left, "\rho"] \arrow[r, shift right, leftarrow, "\iota" swap] \arrow[loop left, "k"] \& (H^*(A),0)\quad\text{with}\quad \rho \iota=Id.
\end{tikzcd}
\]
\end{theo}
\begin{proof}
Consider the short exact sequences 
\[0\rightarrow B^n(A)\rightarrow Z^n(A)\rightarrow H^n(A)\rightarrow 0\text{ and }
0\rightarrow Z^n(A)\rightarrow A^n\rightarrow B^{n+1}(A)\to 0
\] 
Choosing splittings of both sequences we obtain a decomposition
\[A^n=B^n(A)\oplus H^n(A)\oplus B^{n+1}(A).\]
We let $\iota(a):=(0,a,0)$, $\rho(b,a,b')=a$ and 
$k(b,a,b')=(0,0,b).$
\end{proof}

For the filtered version, we will replace the cohomology of a complex by the first page of the associated spectral sequence, together with a possibly non-trivial differential that will be defined inductively over decreasing weights. 

\begin{theo}
 Let $(A,d,F)$ be a filtered complex. There is a filtered homotopy transfer diagram  
\[
\begin{tikzcd}[ampersand replacement = \&]
(A,d,F) \arrow[r, rightarrow, shift left, "f"] \arrow[r, shift right, leftarrow, "g" swap] \arrow[loop left, "h"] \& (M,d,F)
\end{tikzcd}
\]
where $(M,d,F)$ is a filtered complex such that
\[F^pM^n\cong \bigoplus_{q\geq p}E^{q,n-q}_1(A)\text{ and }d(F^pM)\subseteq F^{p+1}M.\]
Furthermore, the maps $f$ and $g$ are filtered quasi-isomorphisms.
\label{Kadeishvili1filtrat}
\end{theo}

\begin{proof}Assume that for all $q>p$ we have defined a 
complex $(M_q,d)$ fitting into a homotopy transfer diagram of the form 
\[
\begin{tikzcd}[ampersand replacement = \&]
(F^qA,d) \arrow[r, rightarrow, shift left, "f_q"] \arrow[r, shift right, leftarrow, "g_q" swap] \arrow[loop left, "h_q"] \& (M_q,d)
\end{tikzcd}
\]
where $g_q$ is a quasi-isomorphism and $d(M_q)\subseteq M_{q+1}\subseteq M_q$.
Consider the complex $(Q_p,D)$ given by 
$Q_p^n:=M_{p+1}^{n+1}\oplus F^pA^n$ with $D(m,a):=(-dm,-g_{p+1}(m)+da).$
Apply Theorem \ref{Kadeclassic} to this complex to get a homotopy transfer diagram
\[
\begin{tikzcd}[ampersand replacement = \&]
(Q_p,D) \arrow[r, rightarrow, shift left, "\rho"] \arrow[r, shift right, leftarrow, "\iota" swap] \arrow[loop left, "k"] \& (H^*(Q_p),0)
\end{tikzcd}\quad\text{ with }\quad\rho \iota=Id.
\]
Let $M_p:=M_{p+1}\oplus V^*_p$ where $V^n_p:=H^n(Q_p)$. 
Define $d:V^n_p\to M_{p+1}^{n+1}$ by $d:=\pi_1 \iota$ and 
$g_{p}:V^n_p\to F^pA^n$ by $g_p:=\pi_2 \iota$, where
$\pi_1:Q_p\to M_{p+1}^{n+1}$ and $\pi_2:Q_p\to F^pA^n$
denote the projections to the first and second components respectively. Note that $\pi_1D+d\pi_1=0$ and $d\pi_2-\pi_2D=g_{p+1}\pi_1$. These maps extend the differential on $M_{p+1}$ and the map $g_{p+1}$ respectively. Indeed, on elements of $V^*_p$ we have \[dd=d\pi_1\iota=-\pi_1D\iota=0\quad\text{ and }\quad
dg_p=d\pi_2\iota=g_{p+1}\pi_1\iota+\pi_2D\iota=g_{p+1}\pi_1\iota=g_{p+1}d,                                                                                                                                                                                                                                                                      \]
where we used the fact that $D\iota=0$.
To define  $f_p$ and $h_p$ we split the sequence 
\[0\to F^{p+1}A^n\to F^{p}A^n\to Gr^p_FA^n\to 0.\]
The differential of $F^pA^n\cong F^{p+1}A^n\oplus Gr^p_FA^n$ may then be written as 
\[d(a',a)=(da'+\tau a,da),\text{ where } \tau:Gr^p_FA^n\to F^{p+1}A^{n+1}\text{ satisfies } d\tau+\tau d=0.\] Moreover, the chosen section $s:Gr^p_FA\to F^pA$ satisfies $ds=\tau +sd.$
Consider the map 
$r:Gr^p_FA\lra Q_p$ given by $r(a):=(f_{p+1}\tau a,sa-h_{p+1}\tau a).$
Then $Dr=rd$ and so $r$ is a morphism of complexes.
We define 
$\tilde f:=(\rho -\pi_1 k) r:Gr^p_FA\to M_p$ and let
$\tilde h:=\pi_2 k r:Gr^p_FA\to Gr^p_FA[-1]$.
These maps satisfy the relations
\[d\tilde f-\tilde f d=f_{p+1}\tau\quad\text{ and }\quad d\tilde h+\tilde h d=-g_p\tilde f+s-h_{p+1}\tau.\]
% Indeed, for the first relation we have
% \[\begin{align}d\tilde f-\tilde f d=d(\rho-\pi_1k)r-(\rho-\pi_1k)rd
% = \pi_1\iota\rho r+\pi_1D k r-\rho D r +\pi_1 k Dr=\\
%  =\pi_1 \iota \rho r +\pi_1(Dk+kD)r-\rho D r= \pi_1 r= f_{p+1}\tau.
%  \end{align}
% \]
% For the second relation we have
% \[\begin{align}
% d\tilde h+\tilde h d=-d\pi_2kr-\pi_2krd
% =-(g_p\pi_1+\pi_2D)kr-\pi_2kDr=\\
% =-g_{p+1}\pi_1kr+\pi_2\iota\rho r-\pi_2r
% =-g_p(\pi_1k-\rho)r-\pi_2r
% =-g_pf_p-s+h_{p+1}\tau.
%  \end{align}
% \] Falta canviar signes per tot
Now, for $(a',a)\in F^{p+1}A\oplus Gr^p_FA\cong F^pA$ we define 
\[f_{p}(a',a):=f_{p+1}(a')+\tilde f(a)\text{ and }h_{p}(a,a'):=h_{p+1}(a')+\tilde h(a).\]
The above relations ensure that $f_p$ is a morphism of complexes and that $h_{p}$ is a homotopy from  the identity to $g_{p}f_{p}$.
We next show that $g_p$ is a quasi-isomorphism.
Consider the inclusion map $j:M_{p+1}\to M_p$ and the corresponding mapping cone 
\[C^n(j):=M_{p+1}^{n+1}\oplus M^n_p\quad;\quad D(a,b)=(-da,-a+db).\]
Elements in $Z^nC(j)$ may be written as $(a,b+v)$ where $a,b\in M_{p+1}$ and
$v\in V_p$ satisfy $a=d(b+v)$.
One then gets 
$(a,b+v)+D(b,0)=(a-db,v)=(a',v)$
and so every element in $H^n(C(j))$ has 
a representative of the form $(a',v)$, with $a'\in M_{p+1}$ and $v\in V_p$.
The map
$(Id\times g_p)^*:H^*(C(j))\lra H^*(Q_p)$ given by 
   $[(a',v)]\mapsto[(a',g_p(v))]$
is an isomorphism of graded vector spaces. Indeed,
given $[(m,a)]\in H^*(Q_p)$, there is by definition an element $v\in V_p$ such that $dv=m$ and $g_p(v)=a$. Then $[(m,a)]\mapsto [(dv,v)]$ defines an inverse map.
The maps
\[\begin{tikzcd}[row sep=1cm,column sep=1cm]
\cdots \arrow[r] & H^{n}(M_{p+1})
\arrow[r] \arrow[d,"~Id^*"]  
& H^{n}(M_{p}) 
\arrow[r] \arrow[d,"~g_p^*"] 
& H^{n+1}(C(j))
\arrow[r] \arrow[d,"~(Id\times g_p)^*"] 
& \cdots\\
\cdots \arrow[r] &  H^{n}(M_{p+1})\arrow[r] & H^{n}(F^pA)
\arrow[r] & H^{n+1}(Q_p)
\arrow[r] & \cdots
\end{tikzcd}\]
together with the five lemma
make $g_p^*$ into an isomorphism.

Lastly, let $F^pM:=\bigoplus_{q\geq p}M_q$. The maps $f_p$, $g_p$ and $h_p$ give filtered maps $f,g$ and $h$ satisfying the required properties.
Since $d(F^pM)\subseteq F^{p+1}M$, we have 
\[F^pM^n=\bigoplus_{q\geq p} Gr^p_FM^n=\bigoplus_{q\geq p} E_1^{p,n-p}(M)\cong \bigoplus_{q\geq p} E_1^{p,n-p}(A).\qedhere\]
\end{proof}

\begin{rema}
The differential $d$ of the model $M$ obtained above may be understood as a perturbation of the differential $d_1$ of $E_1$ which takes into account higher weights: if the spectral sequence degenerates at $E_r$, with $r\geq 2$, 
we may write 
\[d=d_1+\tau_2+\cdots +\tau_{r-1}\text{ where }\tau_i:M^{n}_p\to M_{p+i}^{n+1}.\]
If the spectral sequence degenerates at $E_1$ we have $d=0$, and if it degenerates at $E_2$ we have $d=d_1$. 
Likewise, the triple $(f,g,h)$ may be understood as a perturbation of a triple $(\tilde f,\tilde g,\tilde h)$ obtained after applying a bigraded version of Kadeishvili's theory to the complex $(E_0^{*,*}(A),d_0)$. 
These perturbations may be non-trivial even when degeneration occurs at the first page.
\end{rema}

\begin{exam}
When $F$ is the trivial filtration $0=F^{1}A\subset F^{0}A=A$,
 we have $E_{1}^{0,n}=H^n(A)$ and $E_1^{p,n-p}=0$ for all $p\neq 0$ and
$M$ inherits the trivial filtration. In this case, the proof of Theorem \ref{Kadeishvili1filtrat}
reduces verbatim to the proof of Theorem \ref{Kadeclassic}.
\end{exam}

\section{Filtered minimal $A_\infty$-models}
In this section we endow any stage of the spectral sequence associated to a filtered dg-algebra with a compatible $A_\infty$-algebra structure. All results in this section are also true if we consider commutative or Lie dg-algebras instead, giving $C_\infty$- and $L_\infty$-structures respectively. Indeed, the explicit formulae for producing an $A_\infty$-structure out of a homotopy transfer diagram 
coincide with the formulae in the $C_\infty$-case (see Theorem 12 in \cite{ChGe}). In the $L_\infty$-case they are completely analogous (see for instance Theorem 10.3.9 in \cite{LV}).
Assume given a homotopy transfer diagram
\[
\begin{tikzcd}[ampersand replacement = \&]
(A,d) \arrow[r, rightarrow, shift left, "f"] \arrow[r, shift right, leftarrow, "g" swap] \arrow[loop left, "h"] \& (B,d)
\end{tikzcd}
\]
where $(A,d)$ is endowed with a product $\mu:A\otimes A\rightarrow A$ making $(A,d,\mu)$ into a dg-algebra.
Following Theorem 5 of \cite{Markl2:06:geophys}, define the associated \textit{$\mathfrak{p}$-kernels} to be the degree $(n-2)$ linear maps
$\mathfrak{p}_n: A^{\otimes n}\rightarrow A$ , for all $n\geq 1$,
given by 
\[\mathfrak{p}_{n}:=\sum\limits_{\substack{k+\ell=n\\k,\ell\geq 1}}(-1)^{k(\ell+1)}\mu((h\circ \mathfrak{p}_k)\otimes(h\circ \mathfrak{p}_l))\]
with the convention that $h\circ \mathfrak{p}_1=Id$. 
Similarly, the associated \textit{$\mathfrak{q}$-kernels} are the degree $(1-n)$ linear maps 
$\mathfrak{q}_n:A^{\otimes n}\rightarrow A$, for $n\geq 1$,
defined as 
$$\mathfrak{q}_n:=-\mu((h\circ \mathfrak{q}_{n-1})\otimes Id)+\sum_{j=1}^{n-1}(-1)^{jn+n-j^2}\mu(gf_j\otimes(h\circ \mathfrak{q}_{n-j})).$$
with the convention that $\mathfrak{q}_1=Id$, where $$(gf)_m:=gf\circ \mathfrak{q}_m+\sum_{B(m)}(-1)^{u(r_1,\cdots,r_k)}(h\circ \mathfrak{q}_k)((gf\circ \mathfrak{q}_{r_1})\otimes \cdots \otimes (gf\circ \mathfrak{q}_{r_k})),$$
$$B(n):=\{(k,r_1,\cdots,r_k)|2\leq k \leq n, r_1,\cdots,r_k\geq 1, r_1+\cdots+r_k=n\},$$
and
$$u(r_1,\cdots,r_k):=\sum_{1\leq i<j\leq n}r_i(r_j+1).$$

% \begin{exam}
% For small values of $n$ we have:
% \begin{itemize}
%     \item $p_2:A^{\otimes 2}\to A$ has degree 0 and is given by
%     $$p_2=\mu((h\circ p_1)\otimes (h\circ p_1))=\mu.$$
%    \item $q_2:A^{\otimes 2}\to A$ has degree -1 and is given by
%     $$q_2=-\mu((h\circ q_1)\otimes Id)-\mu(gf\otimes(h\circ q_1)).$$
%     \item $p_3:A^{\otimes 3}\to A$ has degree -1 and is given by
%     $$p_3=\mu((h\circ p_1)\otimes (h\circ p_2))+\mu((h\circ p_1)\otimes (h\circ p_2))=\mu(Id\otimes h\circ p_2))+\mu((h\circ p_2)\otimes id).$$
%    \item $q_3:A^{\otimes 3}\to A$ has degree -2 and is given by
%     $$q_3=-\mu(h\circ q_2)\otimes Id)-\mu(gf\otimes (h\circ q_2)),$$
%     where $$gf_2=gf\circ q_2-(h\circ p_2)((gf\circ q_1)\otimes gf\circ q_1))=gf\circ q_2-(h\circ p_2)((gf)\otimes gf)).$$
% \end{itemize}
% \end{exam}

\begin{rema}
The $\mathfrak{p}$-kernels and $\mathfrak{q}$-kernels are also defined when $A$ has the structure of an $A_\infty$-algebra (see Theorem 5 of \cite{Markl2:06:geophys}, see also Section 3 in \cite{Kop} for more details). For simplicity, we have chosen to restrict to the dg-algebra case.
\end{rema}

Assume now we have a filtered homotopy transfer diagram as in Definition \ref{deffilHTD} where $A$ is endowed with a multiplicative structure $\mu:A\otimes A\to A$ compatible with filtrations:
$\mu(F^pA\otimes F^qA)\subseteq F^{p+q}A$.
The following is straightforward:
\begin{lemm}\label{pqrespect} The $\mathfrak{p}$-kernels and $\mathfrak{q}$-kernels are compatible with filtrations:
\[\mathfrak{p}_n(F^{p_1}A\otimes\cdots\otimes F^{p_n}A)\subseteq F^{p_1+\cdots+p_n}A\text{ and }\mathfrak{q}_n(F^{p_1}A\otimes\cdots\otimes F^{p_n}A)\subseteq F^{p_1+\cdots+p_n}A.\]
\end{lemm}

\begin{defi}
A \textit{filtered $A_\infty$-algebra} is an $A_\infty$-algebra $(A,d,\nu_s)$ endowed with a filtration 
$\{F^pA^n\}$ such that $d(F^pA^n)\subseteq F^{p}A^{n+1}$ and 
\[\nu_s(F^{p_1}A^{n_1}\otimes \cdots \otimes F^{p_s}A^{n_s})\subseteq F^{p_1+\cdots+p_s}A^{n_1+\cdots+n_2-s+2}.\] 
\end{defi} 

\begin{defi}
A \textit{filtered $A_\infty$-morphism} $f:A\rightarrow B$ is a morphism of $A_\infty$-algebras compatible with filtrations:
for all $s\geq 1$, the maps $f_s:A^{\otimes s}\to B$ satisfy
\[f_s(F^{p_1}A^{n_1}\otimes\cdots\otimes F^{p_s}A^{n_s})\subseteq F^{p_1+\cdots+p_s}B^{n_1+\cdots+n_s-s+1}.\]
It is said to be a \textit{filtered quasi-isomorphism} if $f_1$ is a filtered quasi-isomorphism.
\end{defi}

\begin{defi}
Let $f,g:A\rightarrow B$ be filtered $A_\infty$-morphisms.
A filtered $A_\infty$-homotopy from $f$ to $g$ is an $A_\infty$-homotopy 
$h$
compatible with filtrations: 
\[h_s(F^{p_1}A^{n_1}\otimes\cdots\otimes F^{p_s}A^{n_s})\subseteq F^{p_1+\cdots+p_s}B^{n_1+\cdots+n_s-s}\text{ for all }s\geq 1.\]
\end{defi}

\begin{prop}
Consider a filtered homotopy transfer diagram \[
\begin{tikzcd}[ampersand replacement = \&]
(A,d,F) \arrow[r, rightarrow, shift left, "f"] \arrow[r, shift right, leftarrow, "g" swap] \arrow[loop left, "h"] \& (B,d,F)
\end{tikzcd}
\]
where $(A,d,F)$ is a filtered dg-algebra. For $s\geq 2$, let
$$\nu_s:=f\circ \mathfrak{p}_s \circ g^{\otimes s}\quad,\quad\Ff_s:=f\circ \mathfrak{q}_s\quad,\quad\Gg_s:=h\circ \mathfrak{p}_s\circ g^{\otimes s}\quad\text{ and }\quad \Hh_s:=h\circ \mathfrak{q}_s.$$ Where $\mathfrak{p}_s$ and $\mathfrak{q}_s$ are the associated $\mathfrak{p}$-kernels and $\mathfrak{q}$-kernels.
Then:
\begin{enumerate}
    \item $(B,d,\nu_s,F)$ is a filtered $A_\infty$-algebra,
    \item $\Ff=(f,\Ff_s):A\rightarrow B$ 
    and $\Gg=(g,\Gg_s):B\rightarrow A$ are
    filtered $A_\infty$-morphisms,
   \item $\Hh=(h,\Hh_s)$ is a filtered $A_\infty$-homotopy
   from the identity to $\Gg\Ff$.
\end{enumerate}
\label{FHTT}
\end{prop}
\begin{proof}By Lemma \ref{pqrespect}, $\mathfrak{p}$-kernels and $\mathfrak{q}$-kernels are compatible with filtrations, as are $f$, $g$, and $h$ by assumption. It therefore suffices to apply the classical non-filtered proof (see Theorem 5 in \cite{Markl2:06:geophys}, see also Section 3 of \cite{Kop}).
\end{proof}
 
\begin{exam}\label{exemplet}Consider the free filtered dg-algebra 
 $A=\Lambda(a,b,c,e)$ generated by
$a,b,c\in F^0A^1$ and $e\in F^1A^1$, with 
$da=(b-e)c,\, db=de=dc=0.$
For simplicity, we will only analyze degrees $\leq 2$. We have $H^1\cong \langle [b], [c], [e]\rangle$ and
\[H^2(A)\cong \langle x:=[ac], y:=[bc]=[ec],  z:=[a(b-e)], w:=[be]\rangle,\]
with ring structure $[b]\cdot [c]=[e]\cdot [c]=y$ and $[b]\cdot [e]=w.$
The first page is
\[E_1^{0,1}(A,F)\cong \langle  [b], [c]\rangle\text{, } E_1^{1,0}(A,F)\cong \langle [e]\rangle \]
\[E_1^{0,2}(A,F)\cong \langle  x':=[ac],  z':=[ab]\rangle\text{ and } E_1^{1,1}(A,F)\cong \langle  y':=[ec], w':=[be]\rangle.\]
The multiplicative structure induced naturally on $E_1$ gives:
\[[b]\cdot[c]=0, \quad[e]\cdot [c]=y'\quad \text{ and }\quad[b]\cdot [e]=w'.\]
This spectral sequence degenerates at $E_1$. Still, products in $E_1$ differ from those of $H^*(A)$, where $[b]\cdot[c]\neq 0$.

To compute the filtered $A_\infty$-model consider the map $f:A\to E_1(A)$ given by 
$b\mapsto [b]$, $c\mapsto [c]$, $e\mapsto [e]$ and $a\mapsto 0$.
The condition $f(da)=df(a)=0$ implies  
$f(bc):=y'$. This leads to a well-defined filtered homotopy transfer diagram.
We obtain
$\nu_2([b],[c]):=f (b\cdot c)=y'$
so we recover the ``missing'' product.
\end{exam}

\begin{rema}
The above example shows that even when the spectral sequence degenerates at $E_1$, for which the differential of the model is trivial,
bidegrees are not necessarily preserved in the multiplicative structure. In general, we have:
\[\nu_s(E_1^{p_1,q_1}\otimes\cdots\otimes E_1^{p_s,q_s})\lra \bigoplus_{k\geq 0} E_1^{p_1+\cdots+p_s+k,q_1+\cdots+q_s+2-s-k}\text{ for all }s\geq 2.\]
We will see in Section \ref{SecMHT} how, in the setting of mixed Hodge theory, multiplicative structures always preserve bidegrees.
\end{rema}

Kadeishvili's theory can be summarized by saying that every dg-algebra $(A,d)$ has a \textit{minimal $A_\infty$-model}: there
is an $A_\infty$-algebra structure on $H^*(A)$, with trivial differential, together with an $A_\infty$-morphism $H^*(A)\to A$ which is a quasi-isomorphism. Moreover, such minimal $A_\infty$-model is unique up to $A_\infty$-isomorphism. We next state a filtered version of these results.

\begin{defi}
We say that a filtered $A_\infty$-algebra $(M,d,\nu_i,F)$ is \textit{filtered minimal} if 
$M$ has a bigrading with $F^pM^n=\bigoplus_{q\geq p}M^{q,n-q}$
and $d(F^pM^n)\subseteq F^{p+1}M^{n+1}$.
\end{defi}

\begin{prop}\label{miniso}
Every filtered quasi-isomorphism between filtered minimal 
$A_\infty$-algebras is an isomorphism.
\end{prop}
\begin{proof}Let $f:A\to B$ be a filtered quasi-isomorphism between filtered minimal 
$A_\infty$-algebras.
Since $A$ is filtered minimal we have
$A^{p,q}=E_0(A)^{p,q}=E_{1}^{p,q}(A)$ and $d_0=0$,
and similarly for $B$. Moreover, $f$ induces isomorphisms
 $E_1(A)\cong E_1(B)$. This gives $A^{p,q}=E_{1}^{p,q}(A)\cong E_{1}^{p,q}(B)=B^{p,q}$.
\end{proof}

\begin{defi}
 A \textit{filtered minimal $A_\infty$-model} of a filtered dg-algebra $(A,d,F)$ is a filtered minimal $A_\infty$-algebra $M$ together with a filtered $A_\infty$-morphism $M\to A$ which is a filtered quasi-isomorphism.
\end{defi}

The existence of filtered minimal $A_\infty$-models for filtered dg-algebras is now a consequence of Theorem \ref{Kadeishvili1filtrat} together with Proposition \ref{FHTT}.

\begin{theo}\label{r0}Every filtered dg-algebra has a filtered minimal $A_\infty$-model.
\end{theo}

\begin{rema}
There is also a filtered version of minimal models in the rational homotopy context, 
due to Halperin and Tanr\'{e} \cite{HT}. They show that, given a filtered commutative dg-algebra $(A,d,F)$, one may take a Sullivan minimal bigraded model $\rho:(M,d)\to (E_0(A),d_0)$ and find perturbations $\tau$ and $\rho'$ of the differential on $M$ and of $\rho$ such that 
$\rho+\rho':(M,d+\tau)\lra (A,d)$
is a Sullivan minimal model compatible with filtrations and inducing an isomorphism at $E_1$. The filtration on $M$ is given by the column filtration.
\end{rema}

To obtain filtered minimal $A_\infty$-models on any stage of the associated spectral sequence we will relax the compatibility conditions of the filtrations with respect to higher operations as follows.

\begin{defi}\label{defErmin}
Let $r\geq 0$.
An \textit{$E_r$-minimal $A_\infty$-model} of a filtered dg-algebra $(A,d,F)$ is an $A_\infty$-algebra $(M,d,\nu_i)$
 with a bigrading $M=\bigoplus M^{p,q}$ such that
 \[d(M^{p,q})\subseteq \bigoplus_{k\geq 0}M^{p+r+k+1,q-r-k}\quad \text{ and for all }s\geq 2,\]
 \[\nu_s(M^{p_1,q_1}\otimes\cdots\otimes M^{p_s,q_s})\subseteq \bigoplus_{k\geq 0}M^{p_1+\cdots+p_s+(2-s)r+k,q_1+\cdots+q_s+(2-s)(1-r)-k},\]
 together with a morphism $f:M\to A$ of $A_\infty$-algebras satisfying
 \[f_s(M^{p_1,q_1}\otimes\cdots\otimes M^{p_s,q_s})\subseteq F^{p_1+\cdots+p_s+(1-s)r}A^{p_1+\cdots+p_s+q_1+\cdots+q_s-s+1}\]
 and such $E_r(f_1):E_{r}(M)\to E_{r}(A)$ is a quasi-isomorphism of complexes.
\end{defi}

\begin{rema}
 The case $r=0$ corresponds with the notion of filtered minimal $A_\infty$-model.
A proof analogous to that of Proposition \ref{miniso} shows that any two $E_r$-minimal models of a filtered dg-algebra are isomorphic.
\end{rema}

\begin{theo}\label{higherr}Let $r\geq 0$. Every filtered dg-algebra has an $E_r$-minimal $A_\infty$-model.
\end{theo}
\begin{proof}
The case $r=0$ is Theorem \ref{r0}.
Assume it is proven for all $s<r$.
Consider the d\'{e}calage filtration
$\Dec F^pA^n:=\{x\in F^{p+n}A^n; dx\in F^{p+n+1}\}=Z_1^{p+n,-p}.$
Then the triple $(A,d,\Dec F)$ is a filtered dg-algebra and by Proposition I.3.4 of \cite{DeHII} we have isomorphisms
$E_s^{p,n-p}(A,\Dec F)\cong E_{s+1}^{p+n,-p}(A,F)$ for every $s>0$.
By induction hypothesis we may take an $E_{r-1}$-minimal $A_\infty$-model of $(A,d,\Dec F)$, so there is an $A_\infty$-morphism 
$\psi:(M,d,\nu_i)\lra (A,d,\mu)$
with a bigrading $M=\bigoplus \widetilde M^{p,q}$ satisfying the conditions of Definition \ref{defErmin} for $(r-1)$ and the filtration $\Dec F$ of $A$.
Consider the shifted bigrading
$M^{p,n-p}:=\widetilde M^{p-n,2n-p}.$
There are isomorphisms
$E_{s+1}^{p+n,-p}(M)\cong E_s^{p,n-p}(\widetilde M).$
One verifies that this new bigrading satisfies the conditions of Definition \ref{defErmin} for $r$
and the filtration $F$ of $A$.
\end{proof}

\begin{rema}
 Assume that the spectral sequence of a filtered dg-algebra $(A,d,F)$ degenerates at $E_{r+1}$, for some $r\geq 0$.
 Then the $E_r$-minimal $A_\infty$-model has $d=0$ and since $E_{r+1}(A)\cong H^*(A)$, this gives an $A_\infty$-structure on cohomology satisfying
 \[\nu_s(E_{r+1}^{p_1,*}\otimes\cdots\otimes E_{r+1}^{p_s,*})\subseteq \bigoplus_{k\geq 0}E_{r+1}^{p_1+\cdots+p_s+(2-s)r+k,*}
 \quad\text{ for all }s\geq 2.
 \]

\end{rema}

\section{Complex manifolds}
The complex de Rham
algebra of every complex manifold $X$ admits a bigrading
\[\Aa^k:=\Aa_{dR}^k(X)\otimes_\RR\CC=\bigoplus_{p+q=k} \Aa^{p,q}\] by forms of type $(p,q)$
and the exterior differential decomposes as 
$d=\delb+\del$
where $\delb$ has bidegree $(0,1)$ and $\del$ is its complex conjugate, of bidegree $(1,0)$. The \textit{Hodge filtration} is the decreasing filtration of $\Aa^*$ defined by the first degree:
$$F^p\Aa^k:=\bigoplus_{q\geq p} \Aa^{q,k-q}.$$
The \textit{Fr\"{o}licher spectral sequence} is the
spectral sequence associated to the Hodge filtration. Its $E_0$-page is the Dolbeault algebra 
$(E_0^{*,*},d_0)=(\Aa^{*,*},\overline{\partial})$
and its $E_1$-page is isomorphic to Dolbeault cohomology 
\[E_1^{p,q}(X)\cong H_{\overline{\partial}}^{p,q}(X):={\Ker(\delb)\over\Img(\delb)}|_{(p,q)}.\]
This spectral sequence converges to complex de Rham cohomology $H^*_{dR}(X)\otimes_\RR\CC$.

Note that $H_\delb^{*,*}(X)$ naturally inherits a ring structure and, in fact, via a bigraded homotopy transfer 
diagram
\[
\begin{tikzcd}[ampersand replacement = \&]
(\Aa^{*,*},\delb) \arrow[r, rightarrow, shift left, "f"] \arrow[r, shift right, leftarrow, "g" swap] \arrow[loop left, "h"] \& (H_\delb^{*,*}(X),0)
\end{tikzcd}.
\]
it has an $A_\infty$-structure which we call the \textit{Dolbeault $A_\infty$-structure}. This structure is purely geometric and its associated operations satisfy 
\[\nu_s^{\mathrm{Dol}}:H_\delb^{p_1,q_1}(X)\otimes\cdots\otimes H_\delb^{p_s,q_s}(X)\lra H_\delb^{p_1+\cdots+p_s,q_1+\cdots+q_s+2-s}(X).\]
On the other hand, 
Theorem \ref{r0} gives a filtered $A_\infty$-structure on the same bigraded vector space $H_\delb^{*,*}(X)$, which is faithful to the topological information and satisfies 
\[\nu_s^{\mathrm{HdR}}:H_\delb^{p_1,q_1}(X)\otimes\cdots\otimes H_\delb^{p_s,q_s}(X)\lra\bigoplus_{k\geq 0} H_\delb^{p_1+\cdots+p_s+k,q_1+\cdots+q_s+2-s-k}(X).\]
We call this the \textit{Hodge-de Rham $A_\infty$-structure}.
These two structures may differ substantially, as we exhibit in the following examples.

\begin{exam}Associated to any compact self-dual 4-manifold $M$ there is a compact complex three-dimensional manifold $Tw(M)$, known as its twistor space (see for instance \cite{AHS}). 
The twistor space $Z=Tw(\TT^4)$ of the 4-torus $\mathbb{T}^4$ is homeomorphic to $\mathbb{T}^4\times S^2$ and inherits a complex non-K\"{a}hler structure with no holomorphic forms, so $H_\delb^{1,0}(Z)=0$. On the other hand, there are four classes in $H_\delb^{0,1}(Z)$ (corresponding with the four generators of $\TT^4$)
which, by bidegree reasons, multiply to 0 (see \cite{EaSi}, which includes a detailed description of the Fr\"{o}licher spectral sequence of $Z$).
As a consequence, the Dolbeault multiplicative structure is extremely poor in comparison with the (topological) Hodge-de Rham structure, for which the product of four classes in $H_\delb^{0,1}(Z)$ recovers the top class of $\TT^4$.
Note that this space is formal, and so the higher multiplicative structures $\nu_i$ with $i\geq 3$ are trivial.
\end{exam}

\begin{exam}
 For a Stein manifold $X$, the Dolbeault algebra 
 $(A^{*,*},\delb)$ is quasi-isomorphic (as a commutative dg-algebra) to $(H^{*,0}_\delb(X),0)$. In particular, the Dolbeault $A_\infty$-structure on $H_\delb^{*,0}(X)$ has no higher operations (see Section 4 of \cite{NT}).
 Note, however, that Stein manifolds need not be formal and so, in general, 
 higher operations of the Hodge-de Rham $A_\infty$-structure on $H_\delb^{*,*}(X)$
 will be non-trivial.
\end{exam}

Let us now see a particular example of the Dolbeault and Hodge-de Rham $A_\infty$-structures on a well-known complex non-formal manifold.

\begin{exam}\label{KT}
The Kodaira-Thurston manifold is a $4$-dimensional nilmanifold given by
$
KT:=H_\ZZ \times \ZZ \setminus H \times \RR
$,
where $H$ is the $3$-dimensional Heisenberg Real Lie group and $H_\ZZ$ is the integral subgroup. 
The Lie algebra $\mathfrak{g}$ of $H \times \RR$ is spanned by $X,Y,Z,W$ with bracket $[X,Y] = -Z$.
The Chevalley-Eilenberg algebra associated to $\mathfrak{g}$ is a de Rham model of $KT$ (see for instance Theorem 3.18 of \cite{FOT}).
This is given by the free commutative dg-algebra $\Lambda(x,y,z,y)$ whose generators $x,y,z,w$ form a dual basis of $X,Y,Z,W$, with differential $dz = xy$. 
Moreover, the endomorphism of $\mathfrak{g}$ given by $J(X)=Y$ and $J(Z)=-W$ descends to a complex structure on $KT$. 
Consider the basis $A=X-iJX$ and $B=Z-iJZ$. In the dual basis, we have
\[\Aa:=\Lambda(a,\ov a,b,\ov b)\quad\text{ with }\quad db=\delb b=ia\ov a\text{ and } d\ov b=\del \ov b=-ia\ov a.\]
Here $a,b\in \Aa^{1,0}=F^1\Aa^1$ and $\ov a,\ov b\in \Aa^{0,1}\subseteq F^0\Aa^1$.
On products of 1-forms we have 
\[d(b\ov b)=ia\ov a(b+\ov b)\quad\text{ and so }\quad \delb(b\ov b)=ia\ov a\ov b\quad\text{ and }\quad \del(b\ov b)=ia\ov a b.\]
There is an inclusion of $\Aa$ into the complex de Rham algebra of $KT$ which preserves the bidegrees and induces an isomorphism in Dolbeault cohomology.
We obtain:
\[E_1^{*,*}(KT)\cong 
\arraycolsep=4pt\def\arraystretch{1.4}
 \begin{array}{|c|c|c|c|}
 \hline
[\ov a \ov b]&[b \ov a \ov b ]&[ab\ov a\ov b] \\
  \hline
 [\ov a], [\ov b]&[a\ov b], [b \ov a ]&[ab\ov a ], [ab\ov b] \\
   \hline
 1&[a]&[ab]\\ 
 \hline
\end{array}
\]
\vspace{.2cm}

This spectral sequence degenerates at the first stage.
The Dolbeault $A_\infty$-structure is easily computed. Define a bigraded homotopy transfer diagram
\[
\begin{tikzcd}[ampersand replacement = \&]
(\Aa^{*,*},\delb) \arrow[r, rightarrow, shift left, "f"] \arrow[r, shift right, leftarrow, "g" swap] \arrow[loop left, "h"] \& (H_\delb(KT),0)
\end{tikzcd}
\]
by letting $f(a\ov a):=0$, $f(a\ov a\ov b):=0$, $f(\gamma):=[\gamma]$ for all $\gamma\neq a\ov a,a\ov a\ov b$ with $\delb\gamma=0$ and $f(\gamma)=0$ otherwise, where $\gamma$ denotes a generator of $\Aa^{*,*}$ as a  complex. Let $g([\gamma]):=\gamma$ for all $[\gamma]\in E_1^{*,*}$.
Define $h$  by 
\[h(a\ov a)=-ib\quad\text{ and }\quad h(a\ov a\ov b)=-ib\ov b,\]
with $h(\gamma)=0$ for the remaining generators. With these definitions, the triple $(f,g,h)$ gives a homotopy transfer diagram for the Dolbeault algebra.

There is also a Hodge-de Rham multiplicative structure obtained via a filtered homotopy transfer diagram
\[
\begin{tikzcd}[ampersand replacement = \&]
(\Aa,d,F) \arrow[r, rightarrow, shift left, "f"] \arrow[r, shift right, leftarrow, "g" swap] \arrow[loop left, "h"] \& (H_\delb^{*,*}(\Aa),0,F)
\end{tikzcd}
\]
defined as follows:
let $g[\ov b]:=\ov b+b$ and $g[\gamma]:=\gamma$ for all $\gamma\neq \ov b$, so that $dg=0$.
Also, let  $f(a\ov a):=0$ and $f(a\ov a \ov b):=[ab\ov a]$. Let $f(\gamma)=[\gamma]$ for all $\gamma\neq a\ov a,a\ov a \ov b$, and $f(\gamma)=0$ otherwise. We have $fd=0$ and $dg=0$. Define $h$ by letting 
$h(a\ov a)=-ib$, $h(a\ov a\ov b)=-ib\ov b$ and  zero otherwise.

We obtain the following table for
ordinary products (we only show non-trivial products of degree $\leq 3$) that exhibit 
a non-trivial perturbation in the Hodge-de Rham case: 
\[\def\arraystretch{1.4}
\begin{array}{cc|r|rr}
&\nu_2(-,-)&\text{Dolbeault}&\text{Hodge-de Rham}&\\
\hline
&[a], [\ov b]&[a\ov b]&[a\ov b]+[ab]&\\
&[\ov a],[\ov b]&[\ov a\ov b]&[\ov a\ov b]-[b \ov a]&\\
&[a],[\ov a\ov b]&0&[ab\ov a]&\\
&[\ov a],[a\ov b]&0&-[ab\ov a]&\\
&[\ov b],[\ov a\ov b]&0&[b\ov a\ov b]&\\
&[\ov b],[a\ov b]&0&-[ab\ov b]&\\
\end{array}
\]
There are many non-trivial higher operations. In the lowest degree we have 
\[\nu_3([a],[a],[\ov a])= i[ab]\quad\text{ and }\quad
\nu_3([a],[\ov a],[\ov a])=-i[b\ov a]\]
and these are the same for the Dolbeault and Hodge-de Rham structures.
\end{exam}

The above example shows that even in the case when $E_1=E_\infty$, 
Hodge-de Rham products may not preserve bidegrees. 
There are particular situations for which products behave particularly well, such as in 
the compact K\"{a}hler case, for which all higher operations vanish by \cite{DGMS} and products preserve bidegrees.
In the following example, the Dolbeault and Hodge-de Rham $A_\infty$-structures agree and so products preserve bidegrees, while there are non-trivial higher operations.

% We collect this fact in the form of an immediate proposition.
% 
% \begin{prop}
% Let $X$ be a complex manifold and assume that:
% \begin{enumerate}
%  \item The Frölicher spectral sequence degenerates either at $E_1$ or $E_2$,
%  \item There is a homotopy transfer diagram for its Dolbeault algebra
% \[
% \begin{tikzcd}[ampersand replacement = \&]
% (\Aa^{*,*}(X),\delb) \arrow[r, rightarrow, shift left, "f"] \arrow[r, shift right, leftarrow, "g" swap] \arrow[loop left, "h"] \& (H_\delb^{*,*}(X),0)
% \end{tikzcd}
% \]
% such that $\del f=f\del$, $\del g=g\del$ and $\del h+h\del =0.$
% \end{enumerate}
% Then the Dolbeault $A_\infty$-structure coincides with the Hodge-de Rham $A_\infty$-structure. In particular, the associated Hodge-de Rham algebraic operations satisfy 
% \[\nu_s:H_\delb^{p_1,q_1}(X)\otimes\cdots\otimes H_\delb^{p_s,q_s}(X)\lra H_\delb^{p_1+\cdots+p_s,q_1+\cdots+q_s+2-s}(X).\]
% \end{prop}

\begin{exam}
The Iwasawa manifold $\mathbb{I}_3$ is the compact complex nilmanifold of complex dimension 3 defined by the quotient of the complex 3-dimensional Heisenberg group $H_\CC$ modulo the subgroup of matrices with entries in $\ZZ[i]$.  
As in Example \ref{KT}, one may compute the Frölicher spectral sequence using the Chevalley-Eilenberg algebra associated to the Lie algebra of 
$H_\CC$. This is given by 
\[\Aa:=\Lambda(a, b, c, \ov a,\ov b,\ov c)\quad\text{ with }\quad dc=\del c=-ab\quad \text{ and }\quad d\ov c=\delb \ov c =-\ov a\ov b.\]
Its associated spectral sequence degenerates at $E_2$ and we have 
\[E_1^{*,*}(\mathbb{I}_3)\cong 
\arraycolsep=4pt\def\arraystretch{1.8}
 \begin{array}{|c|c|c|c|c|}
 \hline
 [\ov a\ov b\ov c] & [a\ov a\ov b \ov c], [b \ov a\ov b \ov c], \textcolor{gray}{[c \ov a\ov b\ov c]} & \textcolor{gray}{[ab \ov a\ov b\ov c]}, [ac\ov a\ov b \ov c], [bc\ov a\ov b \ov c] & [abc\ov a\ov b \ov c]\\
 \hline
 \begin{array}{c}
 {[\ov a\ov c]}\\{[\ov b\ov c]}
 \end{array}
 & \begin{array}{c} 
 {[a \ov a \ov c]}, [a\ov b \ov c], [b\ov a\ov c]\\
 {[b\ov b \ov c]}, \textcolor{gray}{[c\ov a\ov c]}, \textcolor{gray}{[c\ov b\ov c]}
 \end{array}
  & 
 \begin{array}{c} 
  \textcolor{gray}{[ab\ov a\ov c]}, \textcolor{gray}{[ab\ov b\ov c]}, [ac\ov a\ov c]\\
  {[ac\ov b \ov c]}, [bc\ov b \ov c], [bc\ov a\ov c]
  \end{array}
  &
  \begin{array}{c} 
  {[abc\ov b\ov c]}\\ {[abc\ov a\ov c]}
  \end{array}
  \\
 \hline 
   \begin{array}{c} 
 {[\ov a]}\\{[\ov b] }
   \end{array}
 &  \begin{array}{c} 
 {[a\ov a]}, [a\ov b], [b\ov a]\\
 {[b\ov b]}, \textcolor{gray}{[c\ov a]}, \textcolor{gray}{[c\ov b]}
  \end{array}
 &  \begin{array}{c} 
 \textcolor{gray}{[ab\ov a]}, \textcolor{gray}{[ab\ov b]}, [ac\ov a]\\
 {[ac\ov b]}, [bc\ov a], [bc\ov b]
 \end{array}
 & \begin{array}{c} 
 {[abc\ov a]}\\{[abc\ov b]}
 \end{array}
 \\
 \hline
 1&[a],[b],\textcolor{gray}{[c]}&\textcolor{gray}{[ab]},[ac],[bc]&[abc]\\
 \hline
\end{array}
\]
where the gray classes are the ones dying at $E_2$.
A Dolbeault homotopy transfer diagram is given by the maps $(f,g,h)$ defined as follows:
define $f$ and $g$ by taking the classes (resp. representatives) shown in the table above.
We then let 
\[h(\ov a\ov b)=-\ov c\quad \text{ and }\quad h(x\ov a\ov b)=(-1)^{|x|+1}x\ov c,\quad \text{ where }\quad 
x=a,b,c,ab,ac,bc,abc.\]
Since $\del f=f\del$, $\del g=g \del$ and $\del h+h\del=0$, these maps define also a filtered homotopy transfer diagram
and so the Dolbeault and Hodge-de Rham multiplicative structures agree on $E_1$.
However, there are three different $A_\infty$-structures on $E_2$:
\begin{enumerate}
 \item [(1)] Consider the structure induced on $E_2\cong H(E_1)$ by the  Hodge-de Rham $A_\infty$-structure on $E_1$.
 \item [(2)] Consider the structure induced by a bigraded homotopy transfer for $(E_1(\mathbb{I}_3),d_1)$.
 \item [(3)] Consider the $E_2$-minimal $A_\infty$-structure given by Theorem \ref{higherr}.
\end{enumerate}
The products $\nu_2$ coincide in all three cases. For $\nu_3$ in low degrees we find:
\[\def\arraystretch{1.4}
\begin{array}{cc|r|r|rr}
&\nu_3(-,-,-)&\quad\quad(\text{S1}) & \quad\quad(\text{S2}) &\quad\quad (\text{S3})&\\
\hline
&[a], [a], [b] & 0 & [ac] & [ac] &\\
&[\ov a],[\ov a], [\ov b] & [\ov a\ov c] & 0 & [\ov a\ov c] &\\
&[a], [b], [b] & 0 & [bc] & [bc] &\\
&[\ov a],[\ov b], [\ov b] & [\ov b\ov c] & 0 & [\ov b\ov c] &\\
\end{array}
\]
and so only the third structure recovers the algebraic structure of $H^*(\mathbb{I}_3)$.
\end{exam}

\section{Mixed Hodge theory}\label{SecMHT}

The interaction of Hodge theory with dg-algebras leads to a very particular situation in which multiplicative structures behave strictly with respect to filtrations. This was first noted in \cite{DGMS} and exploited in the case of compact K\"{a}hler manifolds. In Theorem 4.4 of \cite{BuRu}, Budur and Rubi\'{o} show that the $A_\infty$-structure on the cohomology of a complex algebraic variety is strictly compatible with the weight filtration. Their proof uses the bigraded model (see \cite{Mo}, see also \cite{CG1}) which cannot be extended to study the Hodge filtration.
In this section, we propose an alternative approach which takes care of both the weight and the Hodge filtrations.
Let $\kk\subset \RR$ be a sub-field of the real numbers.

\begin{defi}
A \textit{mixed Hodge structure} on a finite dimensional $\kk$-vector space $V$ is given by an increasing
filtration $W$ of $V$, called the \textit{weight filtration},
together with a decreasing filtration $F$ on $V_\CC:=V\otimes\CC$, called the 
\textit{Hodge filtration}, such that  for all $m\geq 0$, each $\kk$-vector space $Gr_m^WV:=W_mV/W_{m-1}V$ 
is a pure Hodge structure of weight $m$ with the filtration induced by $F$ on $Gr_m^WV\otimes\CC$:
\[Gr_m^WV\otimes\CC=\bigoplus_{p+q=m}V^{p,q}\text{ where }V^{p,q}=F^p(Gr_m^WV\otimes\CC)\cap \overline{F}^q(Gr_m^WV\otimes\CC)=\overline{V}^{q,p}.\]
\end{defi}
Morphisms of mixed Hodge structures are given by morphisms $f:V\to V'$ of $\kk$-vector spaces
compatible with filtrations: \[f(W_mV)\subset W_mV'\text{ and }f_\CC(W_mF^pV_\CC)\subset W_mF^pV_\CC'.\]
By Theorem 2.3.5 of \cite{DeHII}, both maps $f$ and $f_\CC=f\otimes_\kk\CC$ are in fact strictly compatible with filtrations.

Given a mixed Hodge structure $(V,W,F)$ there is a direct sum decomposition $V_\CC=\bigoplus I^{p,q}(V)$, called \textit{Deligne splitting},
such that 
\[W_mV_\CC=\bigoplus_{p+q\leq m} I^{p,q}(V)\text{ and }F^kV_\CC=\bigoplus_{p\geq k}I^{p,q}(V).\]
Deligne's splittings are functorial for morphisms of mixed Hodge structures (see Scholie 1.2.11 of \cite{DeHII}, see also Lemma-Definition 3.4 of \cite{PS}).

The cohomology of a complex of mixed Hodge structures is again a mixed Hodge structure, since differentials are compatible with filtrations and the category of mixed Hodge structures is abelian (see Theorem 2.3.5 of \cite{DeHII}). We have:

\begin{prop}For any complex of mixed Hodge structures $(A,d,W,F)$ there is a bifiltered homotopy transfer diagram 
\[
\begin{tikzcd}[ampersand replacement = \&]
(A_\CC,d,W,F) \arrow[r, rightarrow, shift left, "f"] \arrow[r, shift right, leftarrow, "g" swap] \arrow[loop left, "h"] \& (H^*(A_\CC),0,W,F)
\end{tikzcd}
\]
where $g$ is a bifiltered quasi-isomorphism and $f,g,h$ preserve Deligne's splittings.
\label{Kadeishvilimhs}
\end{prop}
\begin{proof}
Since the category of mixed Hodge structures is abelian,
we have short exact sequences of mixed Hodge structures 
\[0\rightarrow B^n(A)\rightarrow Z^n(A)\rightarrow H^n(A)\rightarrow 0\text{ and }
0\rightarrow Z^n(A)\rightarrow A^n\rightarrow B^{n+1}(A)\to 0.
\]
Since Deligne's splittings are functorial, 
we may split the above sequences in a compatible way, and obtain isomorphisms
\[I^{p,q}(A^n_\CC)\cong I^{p,q}(B^n)\oplus I^{p,q}(H^n)\oplus I^{p,q}(B^{n+1}).\]
It now suffices to continue the proof of Theorem \ref{Kadeclassic} using these spaces to get maps $f$, $g$ and $h$ preserving Deligne's splittings.
\end{proof}

\begin{rema}
Note that in Proposition \ref{Kadeishvilimhs} we do not build a homotopy transfer diagram of complexes of mixed Hodge structures, which would be obstructed by the non-trivial group $\mathrm{Ext}^1$ of extensions (see for instance Section 3.5.1 of \cite{PS}). Indeed, the maps $f$, $g$ and $h$ are only defined over $\CC$ and not over $\kk$. However, since they preserve Deligne's splittings, they are strictly compatible with both filtrations $W$ and $F$ over $\CC$. Note as well that, using the descent of Deligne's splittings from $\CC$ to $\kk$ (see Lemma 4.4 of \cite{CiHo}), one may also build a filtered homotopy transfer diagram for $(A,d,W)$ defined over $\kk$, obtaining maps $(f',g',h')$ preserving such splittings and so strictly compatible with $W$. 
However, the data $(f,g,h)$ and $(f',g',h')$ need not be compatible, in the sense that $f'\otimes_\kk\CC\ncong f$ and similarly for $g$ and $h$.
\end{rema}

Consider now a complex of mixed Hodge structures $(A,d,W,F)$ endowed with a dg-algebra structure
$\mu:A\otimes A\lra A$
compatible with $W$ and $F$ (so that $\mu$ is a morphism of mixed Hodge structures).
We will call this a \textit{mixed Hodge dg-algebra}.
This gives both a filtered dg-algebra $(A,d,W)$ over $\kk$ and a bifiltered dg-algebra $(A_\CC,d,W,F)$ over $\CC$, with $A_\CC=A_\kk\otimes\CC$. Proposition \ref{Kadeishvilimhs} now gives:

\begin{prop}\label{MHalgstrict}
Let $(A,d,W,F)$ be a mixed Hodge dg-algebra. Then the $A_\infty$-structure on $H^*(A_\CC)$ is compatible with Deligne's splittings:  
\[\text{If }H^n(A_\CC)=\bigoplus H^{p,q}_n \text{ then }\nu_s(H^{p_1,q_1}_{n_1}\otimes\cdots\otimes H^{p_s,q_s}_{n_s})\subseteq H^{p_1+\cdots+p_s,q_1+\cdots+q_s}_{n_1+\cdots+n_s-s+2}.\]
In particular, for all $s\geq 2$, the maps $\nu_s$ are strictly compatible with $W$ and $F$:
\[\nu_s(W_pF^qH^*(A_\CC)^{\otimes s})=\Img(\nu_s)\cap W_pF^qH^*(A_\CC),\text{ where }\]
\[
W_pF^qH^*(A_\CC)^{\otimes s}=\bigoplus_{\substack{p_1+\cdots+p_s=p\\q_1+\cdots+q_s=q}}W_{p_1}F^{q_1}H^*(A_\CC)\otimes\cdots\otimes W_{p_s}F^{q_s}H^*(A_\CC).\]
\end{prop}

The natural objects arising from Deligne's theory of mixed Hodge structures on the cohomology of complex algebraic varieties are mixed Hodge complexes as well as their multiplicative versions 
in the commutative setting \cite{Mo, Na} used in rational homotopy as well as in the Lie algebra setting \cite{LCL}, important in deformation theory. Such mixed Hodge complexes are more flexible than complexes of mixed Hodge structures, although their homotopy categories are equivalent. We end this section by recalling how one can restrict to our more rigid setting of complexes of mixed Hodge structures.

\begin{defi}
A \textit{mixed Hodge complex} is given by a filtered cochain complex
$(A_{\kk},d,W)$ over $\kk$, a bifiltered cochain complex $(A_{\CC},d,W,F)$ over $\CC$, 
together with a finite string of filtered quasi-isomorphisms of filtered complexes over $\CC$
$$(A_\kk,d,W)\otimes\CC\xrightarrow{\varphi_1}(A_1,d,W)\xleftarrow{\varphi_2}\cdots\xrightarrow{\varphi_{\l-1}} (A_{\l-1},d,W)\xrightarrow{\varphi_\l} (A_\CC,d,W).$$
In addition, the following axioms are satisfied:
\begin{enumerate}
\item[($\mathrm{H}_0$)] The cohomology $H^*(A_\kk)$ is of finite type.
\item[($\mathrm{H}_1$)] For all $p$, the differential on $Gr_p^WA_\CC$ is strictly compatible with $F$.
\item[($\mathrm{H}_2$)] The filtration $F$ induced on $H^n(Gr^W_pA_{\CC})$, defines a pure Hodge structure of
weight $p+n$ on $H^n(Gr^W_pA_\QQ)$, for all $n$, and all $p$.
\end{enumerate}
Such a mixed Hodge complex is said to be \textit{connected} if the unit map $\kk\to A_\kk$ induces an isomorphism $\kk\cong H^0(A_\kk)$.
\end{defi}

\begin{theo}\label{coroDelpreserve}
Let $\Aa=\{(A_\kk,W),(A_\CC,W,F)\}$ be a connected mixed Hodge complex with a compatible dg-algebra structure. Then, the $A_\infty$-structure on $H^*(A_\CC)$ induced from $A_\CC$ preserves Deligne's splittings and all operations are strictly compatible with $W$ and $F$.
\end{theo}
\begin{proof}
By Theorem 3.17 of \cite{CG1}, there is:
\begin{enumerate}[(i)]
 \item a filtered quasi-isomorphism
$f_\kk:(M_\kk,d,W)\to (A_\kk,d,W)$ of dg-algebras,
\item a collection of filtered quasi-isomorphisms $f_i:(M_i,d,W)\to (A_i,d,W)$, and 
\item a bifiltered quasi-isomorphism 
$f_\CC:(M_\CC,d,W,F)\to (A_\CC,d,W,F)$
\end{enumerate}
such that 
$M_\CC=M_\kk\otimes\CC=M_i$ and the triple $(M,d,\Dec W,F)$ is a mixed Hodge dg-algebra, where $\Dec W$ denotes Deligne's d\'{e}calage filtration.
Since $(M,d,\Dec W,F)$ is a mixed Hodge dg-algebra, by Proposition \ref{MHalgstrict}, the $A_\infty$-structure on $H^*(A_\CC)$ is strictly compatible with the corresponding filtrations $\Dec W$ and $F$. It now suffices to use the identity $\Dec W_pH^n(A_\CC)=W_{p-n}H^n(A_\CC)$.
\end{proof}

\begin{rema}
The above result is also true for connected mixed Hodge complexes with a compatible Lie or commutative dg-algebra structure, giving strictness with respect to $W$ and $F$ for the corresponding $L_\infty$- and $C_\infty$-operations respectively.
Indeed, the proof of Theorem 3.17 in \cite{CG1} is an adaptation of Sullivan's inductive construction of minimal models for commutative dg-algebras to the mixed Hodge setting. The same proof adapts to Lie algebras without complications.
\end{rema}

\begin{exam}If the weight filtration of a multiplicative mixed Hodge complex is trivial 
$0=W_{-1}A^n\subseteq W_0A^n=A^n$
then $H^n(A_\kk)$ is a pure Hodge structure of weight $n$. In this case the components of its Deligne splitting
$H^n(A_\kk)=\bigoplus H^{p,q}_n$ are non-trivial only when $q=n-p$. Therefore we may write
\[H^n(A_\CC)=\bigoplus_{p+q=n}H^{p,q}\text{ with }H^{p,q}=F^p\cap \overline{F}^q\text{ and }F^pH^n(A_\CC)=\bigoplus_{q\geq p}H^{q,n-q}.\]
Then Corollary \ref{coroDelpreserve} implies that $\nu_s=0$ for $s\geq 3$ and 
$\nu_2(H^{p,q}\otimes H^{p',q'})\subseteq H^{p+p',q+q'}.$

For instance, let $X$ be a compact K\"{a}hler manifold. Then its complex de Rham algebra with the Hodge filtration 
defines a multiplicative mixed Hodge complex with trivial weight filtration, and we recover the $A_\infty$-counterpart of the Formality Theorem of \cite{DGMS} in the sense of rational homotopy.
\end{exam}

\begin{exam}
More generally, if $X$ is a complex algebraic variety and $\alpha\neq 0$ is a rational number such that the weight filtration is \textit{$\alpha$-pure}:
$Gr_p^WH^n(X;\CC)=0$ for all $p\neq\alpha n$.
Then we have $\nu_s=0$ for $s\geq 3$. 
This is an $A_\infty$-counterpart of the ``purity implies formality'' result of \cite{CiHo}. 
\end{exam}

\bibliographystyle{amsplain}
\bibliography{bibliografia}



\end{document}